\documentclass[10pt,oneside,a4paper]{amsart}
\usepackage{geometry}
\usepackage[english]{babel}
\usepackage{amsfonts}
\usepackage{amsmath}
\usepackage{amssymb}
\usepackage{graphicx}
\usepackage{caption}
\usepackage{subcaption}
\usepackage{amsthm}
\usepackage{eurosym}
\usepackage[dvipsnames,usenames]{color}
\usepackage{cite}
\usepackage{soul}
\usepackage{tikz}
\usepackage{pgf}
\usepackage{lmodern} %separación de palabras acentuadas
\usepackage[T1]{fontenc}
\title{\textbf{A discrete Hardy's uncertainty principle and discrete evolutions}}
\author{Aingeru Fern\'andez-Bertolin}
\date{\today}
\address{A. Fern\'andez-Bertolin:Departamento de Matem\'aticas,  Universidad del Pa\'is Vasco UPV/EHU, apartado 644, 48080, Bilbao, Spain}
\email{aingeru.fernandez@ehu.eus}
\usepackage{hyperref} % hipervinculos
\newtheorem{teo}{Theorem}[section]

\newtheorem{cor}{Corollary}[section]
\theoremstyle{remark}
\newtheorem{rem}{Remark}[section]

\begin{document}

\begin{abstract}
In this paper we give a discrete version of Hardy's uncertainty principle, by using complex variable arguments, as in the classical proof of Hardy's principle. Moreover, we give an interpretation of this principle in terms of decaying solutions to the discrete Schr\"odinger and heat equations.
\end{abstract}
\maketitle

\section{Introduction}% (brief, i have to say more here)

In functional analysis, uncertainty principles are, in general, results that state that a function and its Fourier transform cannot decay too fast simultaneously. The most studied uncertainty principle, which comes back to Heisenberg, says that
\[
\frac{2}{d}\left(\int_{\mathbb{R}^d}|xf(x)|^2\,dx\right)^{1/2}\left(\int_{\mathbb{R}^d}|\nabla f(x)|^2\,dx\right)^{1/2}\ge \int_{\mathbb{R}^d}|f(x)|^2\,dx,
\]
and, moreover, the equality is attained if and only if $f(x)=Ce^{-\alpha|x|^2/2},$ for $\alpha>0$. In the discrete setting there are some versions of this inequality. For instance, in a previous paper \cite{fb}, we studied an inequality discretizing the position and momentum operators (see also \cite{an,cn,gg} for more references to this uncertainty principle). We saw how we can recover the classical minimizer (i.e. the Gaussian) from the minimizer of the discrete uncertainty principle, which is given in terms of modified Bessel functions
\[
I_k(x)=\frac{1}{\pi}\int_0^\pi e^{z\cos\theta}\cos(k\theta)\,d\theta,\ \ k\in\mathbb{Z}.
\]

In this paper, we are interested in giving a discrete version of Hardy's uncertainty principle \cite{dm,ss} in one dimension
\[
|f(x)|\le Ce^{-x^2/2\alpha},\ \ |\hat{f}(\xi)|\le Ce^{-\xi^2/2\beta}, \ \text{with }\alpha\beta<1\Rightarrow f\equiv0,
\]
and, in the case $\alpha\beta=1$ then $f(x)=Ce^{-x^2/2\alpha}.$ Thus, this uncertainty principle is telling us that a function and its Fourier transform cannot have both Gaussian decay for some coefficients $\alpha,\beta$ The original proof of this principle is strongly based on complex variable arguments (Phragmen-Lindel\"of's principle and Liouville's theorem). Moreover, if we consider a solution to the one-dimensional free Schr\"odinger equation
\[
\left\{\begin{array}{l}\partial_tu=i\partial_{xx}u,\\u(x,0)=u_0(x),\end{array}\right.
\]
we can write the solution as
\[
u(x,t)=\frac{e^{ix^2/4t}}{\sqrt{it}}\left(e^{i(\cdot)^2/4t}u_0\right)^\wedge\left(\frac{x}{2t}\right).
\] 
Basically, this says that the solution is the Fourier transform of the initial datum, so we can write Hardy's uncertainty principle in terms of solutions to the Schr\"odinger equation,
\begin{equation}\label{hsc}
u(x,0)=O(e^{-x^2/2\alpha}),\ \ u(x,1)=e^{i\partial_{xx}}u(x,0)=O(e^{-x^2/2\beta}),\ \ \alpha\beta<4\Rightarrow u\equiv0,
\end{equation}
having a precise expression for $u_0$ in the case $\alpha\beta=4.$ On the other hand, it can also be stated in terms of solutions to the heat equation, as follows:
\begin{equation}\label{hhe}
f\in L^2(\mathbb{R}),\ e^{x^2/2\delta}e^{\partial_{xx}}f\in L^2(\mathbb{R}) \text{ for some }\delta\le 2 \Rightarrow f\equiv0. 
\end{equation}

Under these terms, there is a series of papers \cite{ekpv1,ekpv2,ekpv3,ekpv4} and \cite{cekpv} where the authors prove Hardy's uncertainty principle for perturbed Schr\"odinger and heat equations just by using real calculus arguments. In the case of the Schr\"odinger equation, they are able to get a result up to the endpoint case. However, for the heat equation their result does not cover the whole case.

Here first we are going to see, using complex variable arguments, a version of Hardy's uncertainty principle in a discrete setting, saying that if a complex function is controlled in some region of the complex plane by a function that will be closely related to the Gaussian, and its Fourier coefficients are controlled precisely by the minimizer we got in \cite{fb} (i.e. by the modified Bessel function $I_k(x)$), then in some cases we have that the function is zero, or we can give a precise expression for the function. This is not exactly stated as Hardy's classical principle, since now, by using Fourier series we relate the sequence to a periodic function in an interval of the form $[-\pi/h,\pi/h]$, while in the continuous case the relation one gets using the Fourier transform is between functions in the real line. Due to this periodicity, just knowing the behavior of the function in that interval does not give us any information.

Once we have this result, we want to give some uncertainty principles in terms of the free Schr\"odinger and the heat equations, concluding that solutions to these equations cannot decay too fast at two different times. We are going to use two different approaches to prove these results. First, we take advantage of what is known in the continuous case, so we need to assume that the initial datum of the discrete equation is similar to a function on the real line. In this case, it is not reasonable to expect that the solution to the discrete equation is identically zero, but it tends to the zero function in some sense as $h$, the mesh step, tends to zero. However, when this mesh step is fixed, we can use the discrete version of the uncertainty principle we mention above in order to see that under some cases this solution is identically zero. As it is pointed out in \eqref{hsc}, the condition on the decay coefficients in the continuous case is $\alpha\beta<4$, while, the discrete approach we are going to study here leads to a more restrictive condition for $\alpha$ and $\beta$, $\alpha+\beta<2$. Although this is a clear difference between the continuous and the discrete case, we are going to see here that this is the sharp condition for our results.

When preparing this manuscript, we learn about a recent and independent result in this direction \cite{jlmp}. There, the authors use complex variable arguments to prove a sharp analog of Hardy's uncertainty principle considering solutions to the discrete Schr\"odinger equation. On the other hand, they also use the real variable approach in \cite{ekpv1,ekpv2,ekpv3,ekpv4} in order to add real-valued time-dependent potentials. There is overlap between their results and those presented here in Section 4, although the proofs of the results are different.

The paper is organized as follows: In Section 2 we give a general result in the discrete case that is related to Hardy's uncertainty principle. In Section 3 we give a version of Hardy's uncertainty principle for the discrete Schr\"odinger and discrete heat equations, by relating these equations to the continuous equations, noticing that we should not expect the solutions to be identically zero using this approach. In Section 4 we use the results of Section 2 to give discrete versions of Hardy's uncertainty principle in the spirit of the results in Section 3 without using information about the continuous case and, moreover, we give some examples of discrete data that prove the sharpness of the decay conditions. Finally, in an appendix we give some examples of nonzero functions that satisfy the hypotheses of Theorem 2.1 when the statement of the theorem does not give any extra information about the function.

\section{A discrete uncertainty principle using complex variable arguments}

For $h>0$, that represents the mesh size, we define the Fourier coefficients of a $\frac{2\pi}{h}$-periodic function $f_h$ in the following way:
\[
\hat{f}_h(k)=\frac{1}{\sqrt{2\pi}}\int_{-\pi/h}^{\pi/h}f(\xi)e^{-i\xi kh}\,d\xi,\ \ \ f_h(x)=\frac{h}{\sqrt{2\pi}}\sum_{k\in\mathbb{Z}}\hat{f}_h(k)e^{ihkx}.
\]

There is a result concerning the extension of this function $f_h$ to the complex plane. Roughly speaking, this result states that a function that decays faster than $e^{-a|x|}, \forall a>0$, can be extended to the complex plane as an entire function. When the Fourier coefficients are of the form $I_k(u/h^2)$ for $u\in\mathbb{C}$ we are in this case, and we have that the corresponding periodic function is
\[
\frac{h}{\sqrt{2\pi}}\sum_{k\in\mathbb{Z}}I_k(u/h^2)e^{ihk(x+i y)}=\frac{h}{\sqrt{2\pi}}e^{u\cos(x+iy)/h^2}.
\]

So we are going to see that if the Fourier coefficients decay faster than the modified Bessel function, and on the other hand the periodic function $f_h$ is controlled by some function that is related to the one we have computed above, then in some cases the function will be zero, according to $|u|$ and the argument of the Bessel function.

\begin{teo}
Assume that $f_h$ is a complex-valued function and that there are $u=re^{i\theta}\in\mathbb{C},\ b\in[0,2\pi),\ \delta\in(0,\pi/2),$ and $s>0$ such that, for all $y\le 0$,
\[\begin{aligned}
&\left|f_h\left(\frac{b-\theta+\frac{\pi}{2}+\delta}{h}+i y\right)\right|\le Ce^{\frac{\Re u}{h^2}\cos(\theta-\frac{\pi}{2}-\delta)\cosh(yh)-\frac{\Im u}{h^2}\sin(\theta-\frac{\pi}{2}-\delta)\sinh(yh)},
\\&\left|f_h\left(\frac{b-\theta-\frac{\pi}{2}-\delta}{h}+i y\right)\right|\le Ce^{\frac{\Re u}{h^2}\cos(\theta+\frac{\pi}{2}+\delta)\cosh(yh)-\frac{\Im u}{h^2}\sin(\theta+\frac{\pi}{2}+\delta)\sinh(yh)},
\end{aligned}\]
and, for all $y\ge0$,
\[\begin{aligned}
&\left|f_h\left(\frac{b+\theta+\frac{\pi}{2}+\delta}{h}+i y\right)\right|\le Ce^{\frac{\Re u}{h^2}\cos(\theta+\frac{\pi}{2}+\delta)\cosh(yh)+\frac{\Im u}{h^2}\sin(\theta+\frac{\pi}{2}+\delta)\sinh(yh)},
\\&\left|f_h\left(\frac{b+\theta-\frac{\pi}{2}-\delta}{h}+i y\right)\right|\le Ce^{\frac{\Re u}{h^2}\cos(\theta-\frac{\pi}{2}-\delta)\cosh(yh)+\frac{\Im u}{h^2}\sin(\theta-\frac{\pi}{2}-\delta)\sinh(yh)}.
\end{aligned}\]

Moreover,  assume that the Fourier coefficients of $f_h$ satisfy
\[
|\hat{f_h}(k)|\le CI_k\left(\frac{1}{sh^2}\right),\ \ \forall k\in\mathbb{Z}.
\]

Then:
\begin{enumerate}
\item $rs<1\Rightarrow$ There are nonzero functions that satisfy the hypotheses. (See the appendix below)
\item $rs=1\Rightarrow f_h(z)=Ce^{\frac{u}{h^2}\cos(zh-b)}$.
\item $rs>1\Rightarrow f_h(z)\equiv 0.$ 

\end{enumerate}
\end{teo}

\begin{rem}
The hypothesis on the periodic function is telling us that we know how the function behaves near the critical lines where the function $e^{\frac{u}{h^2}\cos(zh-b)}$ has the greatest decay.
\end{rem}

\begin{rem}
In the applications of this theorem, we are going to consider $\delta$ close to $\pi/2$.
\end{rem}

\begin{proof}
To begin with, we show that the case $rs>1$ can be proved as a consequence of the case $rs=1$, as in Hardy's original theorem. Indeed, if $r>\frac1s$, the hypotheses of the theorem are satisfied for $\tilde{u}=\frac{e^{i\theta}}{s}$ and $s$, and in this case $\tilde{r}s=1$. To see this, let us consider $y\le0$, so that ($C$ is a constant that may change from line to line)
\[\begin{aligned}
\left|f_h\left(\frac{b-\theta+\frac{\pi}{2}+\delta}{h}+i y\right)\right|&\le Ce^{\frac{\Re u}{h^2}\cos(\theta-\frac{\pi}{2}-\delta)\cosh(yh)-\frac{\Im u}{h^2}\sin(\theta-\frac{\pi}{2}-\delta)\sinh(yh)}
\\&\le Ce^{re^{-yh}\cos(\pi/2+\delta)/2h^2}\le Ce^{e^{-yh}\cos(\pi/2+\delta)/2sh^2}
\\&\le Ce^{\frac{\Re \tilde{u}}{h^2}\cos(\theta-\frac{\pi}{2}-\delta)\cosh(yh)-\frac{\Im \tilde{u}}{h^2}\sin(\theta-\frac{\pi}{2}-\delta)\sinh(yh)},
\end{aligned}\]
since $y$ is negative and $\cos(\pi/2+\delta)$ is also negative. For the other three lines that we have to control, we can use the same argument. Hence, by (2) in the theorem, we have that $f_h(z)=Ce^{\frac{\tilde{u}}{h^2}\cos\big((z-b)h\big)}$. But then the hypotheses of the theorem are not satisfied, unless $C=0\Rightarrow f_h\equiv0$. 

Thus, we only need to prove that when $rs=1$ we can determine completely the function by the properties of the hypotheses. The proof of this fact relies on the maximum principle and Liouville's theorem, so it is quite similar to the original proof of Hardy's uncertainty principle. As we have pointed out above, the condition for the Fourier coefficients implies that $f_h$ is a $2\pi/h$-periodic function and the extension
\[
f_h(z)=\frac{h}{\sqrt{2\pi}}\sum_{k\in\mathbb{Z}}\hat{f_h}(k)e^{ihkz},\ \ \ z=x+iy,
\]
is an entire function, so we have, for all $z$ in the plane
\[
|f_h(z)|\le C h\sum_{k\in\mathbb{Z}}I_k\left(\frac{r}{h^2}\right)e^{-hky}=Che^{\frac{r}{h^2}\cosh (yh)},
\]

We split up a strip of length $2\pi/h$ in six regions following the figures below (being the first one for the upper half plane and the second one for the lower half plane, and the boundary of each region is determined by the dashed red lines and the line $\{\Re z=0\}$):

\begin{tikzpicture}[domain=0:4]
\draw[->] (6,0) -- (18,0) node[right] {$\Re z$};
\draw[dashed,color=red] (10.75,3) -- (10.75,0) node[below] {$\frac{b+\theta}{h}$};
\draw[dashed,color=red] (17,3) --(17,0) node[below]{$\frac{b+\theta+3\pi/2-\delta}{h}$};
\draw[dashed,color=red] (15,3) --(15,0) node[below]{$\frac{b+\theta+\pi/2+\delta}{h}$};
\draw[dashed,color=red] (6.5,3) --(6.5,0) node[below]{$\frac{b+\theta-\pi/2-\delta}{h}$};
\node [rectangle,draw] (a) at (8.625,1.5) {\textcolor{red}1};
\node [rectangle,draw] (a) at (12.875,1.5) {\textcolor{red}2};
\node [rectangle,draw] (a) at (16,1.5) {\textcolor{red}3};
\end{tikzpicture}

\begin{tikzpicture}[domain=0:4]
\draw[->] (6,0) -- (18,0) node[right] {$\Re z$};
\draw[dashed,color=red] (10.75,0) -- (10.75,-3) node[below] {$\frac{b-\theta}{h}$};
\draw[dashed,color=red] (17,0) --(17,-3) node[below]{$\frac{b-\theta+3\pi/2-\delta}{h}$};
\draw[dashed,color=red] (15,0) --(15,-3) node[below]{$\frac{b-\theta+\pi/2+\delta}{h}$};
\draw[dashed,color=red] (6.5,0) --(6.5,-3) node[below]{$\frac{b-\theta-\pi/2-\delta}{h}$};
\node [rectangle,draw] (a) at (8.625,-1.5) {\textcolor{red}4};
\node [rectangle,draw] (a) at (12.875,-1.5) {\textcolor{red}5};
\node [rectangle,draw] (a) at (16,-1.5) {\textcolor{red}6};
\end{tikzpicture}

The procedure is very simple. In each region, we first multiply $f_h$ by a nice function (which will depend on a parameter $\epsilon$) that decays when $y$ tends to infinity in the unbounded part of the region. Then we see that the product is bounded at the boundary of the region, so we can apply the maximum principle and let $\epsilon$ tend to zero. We do not want the sign of the real part of this function to change in each region, so we look for the behavior of $\sin x\cos x$ on $[\pi/2,\pi]$. We illustrate this in region 1, since the procedure is exactly the same. In this region we consider the function
\[
g_\epsilon(x+iy)=f_h(x+iy)e^{-\frac{u}{h^2}\cos\left(zh-b\right)-i\epsilon \cos^2(Tz)},
\]
where $Tz=a_1z+b_1$ is the linear transformation which maps the interval $[\frac{b+\theta-\pi/2-\delta}{h},\frac{b+\theta}{h}]$ into $[\pi/2,\pi]$.

Hence, we have that
\[
|g_\epsilon(x+iy)=|f_h(x+iy)|e^{\frac{-\Re u}{h^2}\cos(xh-b)\cosh(yh)-\frac{\Im u}{h^2}\sin(xh-b)\sinh(yh)+2\epsilon\phi(x,y)},
\]
where $\phi(x,y)=\cos(a_1x+b_1)\sin(a_1x+b_1)\cosh(a_1y)\sinh(a_1y)$.

We can see that at the boundary of the region we are considering, this function is bounded. Moreover, a simple computation shows that $\frac{2a_1}{h}=\frac{\pi}{\pi/2+\delta}>1$, and this fact tells  us that when $y$ tends to infinity, the leading term is given by the $\epsilon$ part, which is negative. Hence the function is bounded when $y$ is large and we can apply the maximum principle to get that $|g_\epsilon(z)|\le C$. Letting $\epsilon$ tend to zero we conclude that, for all $z$ in region 1,
\[
|f_h(z)e^{-\frac{u}{h^2}\cos(zh-b)}|\le C. 
\] 

When we repeat this procedure in the other regions, we have that $|f_h(z)e^{-\frac{u}{h^2}\cos(zh-b)}|\le C$ in the two strips described in the figure above. One is a strip of length $2\pi/h$ of the upper half plane, and the other one is another strip of length $2\pi/h$ of the lower half plane. By periodicity, this estimate holds for all $z$ in the plane, and by Liouville's theorem this implies that $f(z)=Ce^{\frac{u}{h^2}\cos(zh-b)}$.
\end{proof}

In Section 4 we will use this theorem in order to see that a solution to a discrete equation cannot decay faster than some modified Bessel functions at two different times. What we have to see is the evolution of this equation in the periodic setting, by considering that the solution is given as the Fourier coefficients of some periodic function. The behavior at one time and the expression of the solution in the periodic setting will give us bounds for the function in some vertical lines of the complex plane, so we will use these bounds and the decay of the function at the other time to conclude that the function is zero.

\section{A discrete Hardy's theorem based on the classical Hardy's uncertainty principle}

We are going to consider a solution to the Schr\"odinger equation
\begin{equation}\label{sch}
\partial_t f_k^h(t)=i\frac{f_{k+1}^h(t)-2f_k^h(t)+f_{k-1}^h(t)}{h^2},
\end{equation}
where, in order to use the continuous uncertainty principle in the discrete setting, we are going to assume that the initial datum $f_k^h(0)$ is not far from the evaluation of a continuous function $u_0(x)$, that we will use as initial datum for the continuous Schr\"odinger equation. More precisely, we require that there is a function $u_0$ such that
\[
\sum_{k\in\mathbb{Z}}|f_k^h(0)-u_0(k h)|^2<h^\mu,
\]
for some $\mu>0$. In order to be able to evaluate the function, we are going to require $u_0\in H^s(\mathbb{R})$, for some $s>1/2$. Moreover, this is telling us that at any tme $t$, the solution to the discrete Schr\"odinger equation with initial datum $f_k^h(0)$ is not far from the solution to the discrete equation with initial datum $z_k^h(0)=u_0(kh)$, since the $\ell^2$ norm is preserved in time.

Now what we want to see is how we can relate the solution to a discrete equation to a solution to its continuous version. Since we are going to give an $\ell^\infty$ version of the uncertainty principle, we want to estimate the $\ell^\infty$ norm of the difference between the evaluation of the solution to the continuous equation at time $t$ in the mesh $\{kh\ :\ k>0\}$ with initial data $u_0(x)$ and the solution to the discrete equation with initial data $u_0(kh)$. Since $u_0\in H^{s}$ with $s>1/2$, we have that
\[
\int_{\mathbb{R}}(1+|\xi|)^{(2s-1)/4}|\hat{u}_0(\xi)|\,d\xi <+\infty,
\] 
and using similar arguments as in \cite{str}, where this type of equations is studied when the time is also discrete, we can get that
\begin{equation}\label{cdis}
|u(kh,1)-z_k^h(1)|\le h^{(2s-1)/8}\int_{\mathbb{R}}(1+|\xi|)^{(2s-1)/4}|\hat{u}_0(\xi)|\,d\xi .
\end{equation}

 This is the easiest way to talk about convergence of the discrete equation to the continuous one. However, in \cite{iz} the authors give estimates for a vast variety of Banach spaces.

Under these circumstances, we have the following result:

\begin{teo}
Let $f_m^h(t)$ be a solution to the discrete Schr\"odinger equation \eqref{sch}, and assume that there is $u_0\in H^{s}$ with $s>1/2$ such that for some $\mu>0$,
\begin{equation}\label{mesh}
\sum_{k\in\mathbb{Z}}|f_k^h(0)-u_0(k h)|^2<h^\mu.
\end{equation}

If for some $\alpha,\beta$ satisfying $\alpha\beta<4$ we have, for all $k\in\mathbb{Z}$ and $h>0$,
\begin{equation}\label{decay}
|f_k^h(0)|\le C\frac{I_k(\alpha/h^2)}{I_0(\alpha/h^2)},\ \ |f_k^h(1)|\le C\frac{I_k(\beta/h^2)}{I_0(\beta/h^2)},
\end{equation}
then $u_0\equiv0$ and $|f_k^h(t)|$ tends to zero uniformly in $k$ when $h$ tends to zero.
\end{teo}

\begin{proof}
Since we want to use the continuous result, what we need to prove is that $u_0$ and $u(x,1)=e^{i\partial_{xx}}u_0$ are controlled by some Gaussians. We have to distinguish several cases according to $x$.

First, if $x=0$ we have that $|f_0^h(0)-u_0(0)|\le h^{\mu/2}$, so
\[
|u_0(0)|\le |f_0^h(0)|+h^{\mu/2}\le C+h^{\mu/2},
\]
and letting $h$ tend to zero, we conclude that $|u_0(0)|\le C$.

Now, if $0<x\le 1$ we have that $u_0(x)=u_0\left(j \frac{x}{j}\right)$ for all $j\in\mathbb{N}$, and then 
$|f_j^{x/j}(0)-u_0(x)|\le \frac{x^{\mu/2}}{j^{\mu/2}}\le \frac{1}{j^{\mu/2}}$. Hence
\[
|u_0(x)|\le |f_j^{x/j}(0)|+\frac{1}{j^{\mu/2}}\le C\frac{I_j(\alpha j^2/x^2)}{I_0(\alpha j^2/x^2)}+\frac{1}{j^{\mu/2}},
\]
and it was proved in \cite{fb} that this quotient of Bessel functions tend uniformly in $x$ to the Gaussian $e^{-x^2/2\alpha}$, so letting $j$ tend to infinity we have $|u_0(x)|\le Ce^{-x^2/2\alpha}$. Changing $j$ by $-j$ we can argue in the same fashion to get the same result in $-1\le x<0$.

Finally, if $1\le x$ we use that $u_0(x)=u_0\left( \lceil x\rceil^4 j \frac{x}{\lceil x\rceil^4 j}\right)$ for all $j\in\mathbb{N}$, and then, as in the previous case
\[
|u_0(x)|\le |f_{ \lceil x\rceil^4 j}^{x/ \lceil x\rceil^4 j}|+\frac{x^{\mu/2}}{\lceil x\rceil^{2\mu}j^{\mu/2}}\le C\frac{I_{\lceil x\rceil^4j}(\alpha j^2\lceil x\rceil^8/x^2)}{I_{0}(\alpha j^2\lceil x\rceil^8/x^2)}+\frac{1}{j^{\mu/2}},
\]
and we can use the same reasoning as in \cite{fb} to see that this quotient also converges to the Gaussian uniformly in $x$. Actually, this case is easier than the previous one, since in that case one has to show a uniform estimate over compact sets $|x|\le M$ and then see that if $M$ is large, for $M<|x|$ both functions are very small. In this case, we do not need to do this and it can be proved the uniform convergence in one step. Now letting $j$ tend to infinity we get that $|u_0(x)|\le Ce^{-x^2/2\alpha}$ in this region. Changing $j$ by $-j$ we prove the same when $x\le -1$ and gathering all the results we conclude that $|u_0(x)|\le Ce^{-x^2/2\alpha}$ for all $x\in\mathbb{R}$.

Now what we need to see is the evolution of $u_0(x)$ under the continuous Schr\"odinger equation. Since $u_0\in H^{s}$, with $s>1/2$, we  have that the solution to the continuous Schr\"odinger equation at the mesh is similar to the solution to the discrete Schr\"odinger equation when we discretize $u_0(x)$ by taking its values at the mesh, according to \eqref{cdis}. On the other hand, since the initial data $f_m^h(0)$ satisfies \eqref{mesh}, that is, it is close to the evaluation at the mesh of $u_0$, we have that at time $t$ $f_m^h(t)$ is also close to $z_m^h(t)$, since we have
\[
|f_m^h(t)-z_m^h(t)|<\left(\sum_{k\in\mathbb{Z}}|f_k^h(t)-z_k^h(t)|^2\right)^{1/2}=\left(\sum_{k\in\mathbb{Z}}|f_k^h(0)-u_0(kh)|^2\right)^{1/2}<h^{\mu/2}.
\]

Hence, we can do the same as we have done with $u_0(x)$ to see that it is bounded by another Gaussian. For instance, if $0<x\le 1$,
\[\begin{aligned}
|u(x,1)|&=|u\left(j\frac{x}{j},1\right)|\le |z_{j}^{x/j}(1)|+\frac{1}{j^{(2s-1)/8}}\\&\le |f_j^{x/j}(1)|+\frac{1}{j^{(2s-1)/8}}+\frac{1}{j^{\mu/2}}\le C\frac{I_j(\beta j^2/x^2)}{I_0(\beta j^2/x^2)}+\frac{1}{j^{(2s-1)/8}}+\frac{1}{j^{\mu/2}},
\end{aligned}\]
and letting $j$ tend to infinity we get $|u(x,1)|\le e^{-x^2/2\beta}$. Doing the same for the other regions we finally get that $|u(x,1)|\le e^{-x^2/2\beta}$ for all $x\in\mathbb{R}$, so we can use Hardy's uncertainty principle in its version for Schr\"odinger evolutions \eqref{hsc} to conclude that $u\equiv 0.$ In particular, this tells us that $u_0\equiv 0$ and then what we have is that $|f_k^h(t)|<h^{\mu/2}$ for all $k\in\mathbb{Z}$, so the solution to the discrete equation tends to zero uniformly in $k$ when $h$ tends to zero.
\end{proof}

As we see in the statement of the theorem, we do not prove that the discrete solution is $f_k^h(t)\equiv0$, but it is close to zero as $h$ is getting smaller. We can find examples where we see that we cannot expect to have a complete analog version to Hardy's uncertainty principle using this approach. For instance, if we take as initial datum
\[
f_k^h(0)=\frac{I_k(i/2h^2)}{I_0(5/2h^2)}\Rightarrow |f_k^h(0)|\le \frac{I_k(1/2h^2)}{I_0(5/2h^2)}< \frac{I_k(1/2h^2)}{I_0(1/2h^2)},
\]
since $|I_k(z)|\le I_k(|z|)$ and the modified Bessel function $I_0(x)$ is increasing when $0<x<\infty$. Thus $\alpha=\frac{1}{2}$, and, if we solve the discrete Schr\"odinger equation we get at time $t=1$,
\[
f_k^h(1)=e^{-2i/h^2}\sum_{m=-\infty}^\infty \frac{I_m(i/2h^2)}{I_0(5/2h^2)}I_{k-m}(2i/h^2)=e^{-2i/h^2}\frac{I_m(5i/2h^2)}{I_0(5/2h^2)}\Rightarrow |f_k^h(1)|\le \frac{I_m(5/2h^2)}{I_0(5/2h^2)},
\]
so in this case $\beta=\frac{5}{2}$ and $\alpha\beta=\frac{5}{4}<4$. Then by the theorem, if there exists $u_0$ such that \eqref{mesh} is satisfied, $u_0$ has to be identically zero. In this case it is quite easy to see that this happens. Indeed,
\[
\sum_{k\in\mathbb{Z}}|f_k^h(0)|^2=\frac{1}{I_0^2(5/2h^2)},
\]
and seeing the asymptotic behavior of $I_0(t)$ when $t$ is large, we see that \eqref{mesh} is satisfied for any $\mu>0$. Hence we have an example of a nonzero sequence that satisfies the hypotheses of the theorem. However, in the next section we will see that if we relax the condition on the coefficients $\alpha$ and $\beta$,  at least when $h$ is fixed we can give a result that tells us when $f_k^h$ is going to be identically zero.

On the other hand, if we consider as initial datum the sequence
\[
g_k^h(0)=(-1)^k \frac{I_k(1/h^2)}{I_0(1/h^2)},
\]
then in this case we are not going to have a function $u_0$ satisfying \eqref{mesh}, since the subsequences of the odd members and even members are discrete version of the functions $-e^{-x^2/2}$ and $e^{-x^2/2}$ respectively, as we can see in Figure \ref{fig1}. In this case, the solution to the Schr\"odinger equation is 

\begin{figure}[\h]\centering
\includegraphics{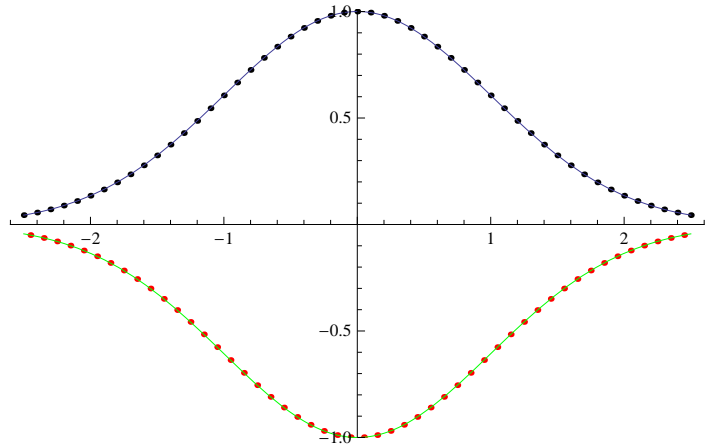}
\caption{Graphic representation of $g_k^h(0)$ when $h=\frac{1}{20}$ and $k\in[-50,50]$ The odd coefficients are plotted in red dots and the even coefficients in black dots. We see the convergence of those subsequences to their respective Guassians $\pm e^{-x^2/2}$ represented in blue and green lines.}\label{fig1}
\end{figure}

\[
g_k^h(1)=e^{-2i/h^2}\sum_{m=-\infty}^\infty(-1)^m \frac{I_m(1/h^2)}{I_0(1/h^2)}I_{k-m}(2i/h^2)=e^{2i/h^2}\frac{I_m((2i-1)/h^2)}{I_0(1/h^2)},
\]
and, at least numerically, the decay coefficient for $g_k^h(1)$ seems to be (see Figure \ref{fig2}) $\beta=5$. This is not very surprising, since the decay one obtains solving the continuous Schr\"odinger equation with initial datum $e^{-x^2/2}$ is precisely $e^{-x^2/10}$, so the coefficients in this case are $\alpha=1,\beta=5$ as well. On the other hand, this is telling us that maybe one can remove the hypothesis \eqref{mesh} on the existence of a proper function $u_0$ and having then that just the decay conditions \eqref{decay} imply that the solution to the discrete Schr\"odinger equation is getting smaller uniformly when $h$ tends to zero.

\begin{figure}[\h]\centering
\includegraphics{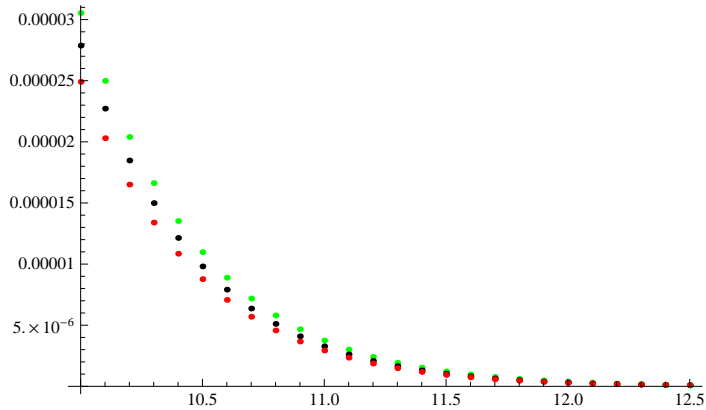}
\caption{Representation of $|g_k^h(1)|$ (black dots) when $h=\frac{1}{20}$ and $k\in[200,250]$ We see that it seems to be controlled by $\frac{I_k(\beta/h^2)}{5^{1/4}I_0(\beta/h^2)}$ for $\beta=5$ (green dots), but this is not the case if we take $\beta=4.9$ (red dots).}\label{fig2}
\end{figure}

Now, considering the discrete heat equation, we have similar results. Although in this case since we have a parabolic equation we can get better estimates for the convergence of the discrete system to the continuous equation, we are going to use the same estimates \eqref{cdis} as for the Schr\"odinger equation. Recall that we are comparing the discrete datum with the samples of a continuous function, and this is well-defined when the function belongs to $H^{s},s>1/2$, so although it could be possible to give some estimates depending on the norm of the function in a weaker space ($L^2$ for instance), the function has to have some degree of smoothness that makes useless the advantage of using these better estimates. On the other hand, in Theorem 3.1 we have used an $\ell^2$ estimate for the difference between the discrete initial datum and the continuous function, while now we can replace it by an $\ell^\infty$ estimate. Indeed, if we assume that a sequence $v_m^h$ satisfies $|v_m^h|\le h^\mu$ for some $\mu>0$ and for all $m\in\mathbb{Z}$, then, solving the discrete heat equation we get
\[
e^{t\Delta_d}v_m^h=e^{-2t/h^2}\sum_{k=-\infty}^\infty v_k^h I_{m-k}(2t/h^2)\Rightarrow |e^{t\Delta_d}v_m^h|\le h^\mu e^{-2t/h^2}\sum_{k=-\infty}^\infty I_{m-k}(2t/h^2)=h^\mu,
\]
so we have that the evolution of $v_m^h$ also satisfies the same estimate. In the case of the Schr\"odinger equation we do not have this property, and this is the reason why we have to consider an $\ell^2$ estimate, since we have the conservation of this norm.  Then the result we have is:
\begin{teo}
Let $v_m^h(t)$ be a solution to the discrete heat equation, and assume that there is $u_0\in H^{s}$ with $s>1/2$ such that for some $\mu>0$,
\begin{equation}\label{mesh2}
\sup_{k\in\mathbb{Z}}|v_k^h(0)-u_0(k h)|<h^\mu.
\end{equation}

If for some $\alpha$ satisfying $\alpha<2$ we have, for all $k\in\mathbb{Z}$ and $h>0$,
\begin{equation}\label{decay2}
|v_k^h(1)|\le C\frac{I_k(\alpha/h^2)}{I_0(\beta/h^2)},
\end{equation}
then $u_0\equiv0$ and $|v_k^h(t)|$ tends to zero uniformly in $k$ when $h$ tends to zero.
\end{teo}

\begin{proof}
Since we have that $u_0\in H^{s}$ we only need to see if $e^{x^2/2\delta}u(x,1)$ belongs to $L^2(\mathbb{R})$ for some $\delta\le 2$. Since we have the uniform estimate \eqref{mesh2} and the convergence of the solution to the discrete heat equation with initial datum $z_k^h=u_0(kh)$ to the solution to the continuous heat equation with initial datum $u_0(x)$ \eqref{cdis}, we can repeat the procedure we use in Theorem 3.1 to conclude that $|u(x,1)|\le e^{-x^2/2\alpha}$. Now, since $\alpha<2$, que can take $\delta$ such that $\alpha<\delta\le 2$ and then
\[
\int_{\mathbb{R}}e^{x^2/\delta}|u(x,1)|^2\,dx\le \int_{\mathbb{R}}e^{x^2(\frac{1}{\delta}-\frac{1}{\alpha})}<+\infty,
\]
so, by Hardy's uncertainty principle for heat evolutions \eqref{hhe} we conclude that $u_0\equiv0$.
\end{proof}

\section{A weaker discrete version of Hardy's uncertainty principle}

As in the classical result of Hardy, we want to have decay conditions on solutions to the discrete Schr\"odinger equation that make the solution to be identically zero. For this purpose, we are going to use results coming  from Section 2, where we use complex variable arguments to conclude that if a periodic function and its Fourier coefficients satisfy certain decay conditions, then the function is identically zero.

The first question is to choose the parameters $b$ and $\theta$ in Theorem 2.1 properly. For that, we rewrite the discrete Schr\"odinger equation \eqref{sch} in a periodic setting, thinking of $f_k^h(t)$ as the Fourier coefficients of a function $g_h(t)$. It is easy to see then that solving \eqref{sch} is the same as solving
\[
\partial_tg_h(x,t)=\frac{2i(\cos(xh)-1)}{h^2}g_h(x,t)\Rightarrow g_h(x,t)=e^{2i(\cos(xh)-1)/h^2}g_h(x,0),
\]
and then we observe that solving the discrete Schr\"odinger equation is basically multiply by the exponential $e^{i\cos(xh)}$, so we are going to take $b, \theta$ so that the conclusion of Theorem 2.1 is that if $rs=1$, then the function is a constant times $e^{\frac{ir}{h^2}\cos(zh)}$, that is, we take $b=0$ and $\theta=\frac{\pi}{2}$. Then we can rewrite Theorem 2.1 in order to have:

\begin{cor}
Assume that $g_h$ is a complex-valued function and that there are $r,s>0,\delta\in(0,\pi/2)$ such that,
\[\begin{aligned}
&\left|g_h\left(\frac{\delta}{h}+i y\right)\right|,\left|g_h\left(\frac{\pi-\delta}{h}+iy\right)\right|\le Ce^{\frac{r}{h^2}\sin\delta\sinh(yh)},\ \text{for all }y\le0\\
&\left|g_h\left(\frac{-\delta}{h}+i y\right)\right|,\left|g_h\left(\frac{\pi+\delta}{h}+iy\right)\right|\le Ce^{-\frac{r}{h^2}\sin\delta\sinh(yh)},\ \text{for all }y\ge0\\
&|\hat{g}_h(k)|\le CI_k\left(\frac{1}{sh^2}\right),\ \ \forall k\in\mathbb{Z}.\end{aligned}
\]

Then:
\begin{enumerate}
\item $rs<1\Rightarrow$ There are nonzero functions that satisfy the hypotheses.
\item $rs=1\Rightarrow g_h=Ce^{\frac{ir}{h^2}\cos(zh)}$ (and $\hat{g}_h(k)=CI_k(ir/h^2)$).
\item $rs>1\Rightarrow g_h\equiv 0.$ 

\end{enumerate}
\end{cor}

This is the version we are going to use in order to have a discrete version of Hardy's theorem. As we have pointed out in Section 3, it is not reasonable to think that the condition on the coefficients $\alpha$ and $\beta$ that measure the decay of the solution at each time will be the same as in the classical result, so we can only talk about weaker versions in this setting. More precisely, we have the following result:

\begin{teo}
Let $f_m^h(t)$ be a solution to the discrete Schr\"odinger equation \eqref{sch}, and assume that $|f_k^h(0)|\le \frac{I_k(\alpha/h^2)}{I_0(\alpha/h^2)},\ |f_k^h(1)|\le \frac{I_k(\beta/h^2)}{I_0(\beta/h^2)}$ with $\alpha$ and $\beta$ positive numbers satisfying $\alpha+\beta<2,$ then $f^h=(f_k^h)\equiv0$. 
\end{teo}

\textit{Proof.} We define the function $g_h(x,t)$ such that $\hat{g}_h(k,t)=f_k^h(t).$ As we have seen in Section 2, the decay of the initial datum implies that the periodic function $g_h$ can be extended to a entire function and that 
\[
|g_h(z,0)|\le \frac{h}{\sqrt{2\pi}I_0(\alpha/h^2)}e^{\frac{\alpha}{h^2}\cosh(yh)},
\]
for $z=x+iy$. Therefore, solving the discrete Schr\"odinger equation in this periodic setting, since it is a multiplication by another periodic and entire function, we have that at time $t=1$ the corresponding function $g_h(z,1)$ is going to be periodic and entire as well. Moreover,
\[
|g_h(z,1)|=|e^{\frac{2i}{h^2}\cos(zh)}g_h(z,0)|\le\frac{h}{\sqrt{2\pi}I_0(\alpha/h^2)} e^{\frac{2}{h^2}\sin(xh)\sinh(yh)+\frac{\alpha}{h^2}\cosh(yh)}.
\]

The key point here is how to choose $\delta$ in the corollary in order to hold all the hypotheses. Since $\alpha+\beta<2$, there is $\epsilon$ such that $\alpha+\beta<2-\epsilon$, so $\alpha<2-\epsilon$. For this $\epsilon$ we take $\delta_\epsilon$ close but less than $\pi/2$ such that $\frac{1}{\sin\delta_\epsilon}=1+\frac{\epsilon}{\alpha}.$ and $r_\epsilon=2-\frac{\alpha}{\sin\delta_\epsilon}=2-\alpha-\epsilon>0.$ Now, if $y\le0$,
\[\begin{aligned}
\left|g_h\left(\frac{\delta_\epsilon}{h}+iy,1\right)\right|&\le \frac{h}{\sqrt{2\pi}I_0(\alpha/h^2)}e^{\frac{2}{h^2}\sin\delta_\epsilon\sinh(yh)+\frac{\alpha}{h^2}\cosh(yh)}\\
&= \frac{h}{\sqrt{2\pi}I_0(\alpha/h^2)}e^{\frac{\alpha}{h^2}e^{yh}+\frac{r_\epsilon}{h^2}\sin\delta_\epsilon\sinh(yh)}\le  \frac{h}{\sqrt{2\pi}I_0(\alpha/h^2)}e^{\frac{\alpha}{h^2}+\frac{r_\epsilon}{h^2}\sin\delta_\epsilon\sinh(yh)}, 
\end{aligned}\]

so we have exactly the kind of bounds we use as hypothesis in Corollary 4.1. In order to get the result, we may let the constant $C$ in the corollary depending on $h$ and $\alpha$. However, using the asymptotic expression of $I_0(t)$ we get rid of the dependence in $h$, since we have, at least for $h$ small enough
\[
\frac{h}{\sqrt{2\pi}I_0(\alpha/h^2)}e^{\frac{\alpha}{h^2}}\le \frac{3\sqrt{\alpha}}{2}.
\]

It is easy to check that for the other three lines in the plane we need to consider we can do the same, so $g_h(z,1)$ satisfies the hypotheses of the corollary with $r=r_\epsilon$. On the other hand, the decay of the Fourier coefficients at time 1 implies that $s=\frac{1}{\beta}$ and again the constan $C$ in the statement of the corollary is independent of $h$. Finally, we have that,
\[
r s=\frac{2-\alpha-\epsilon}{\beta}>1\Longleftrightarrow \alpha+\beta<2-\epsilon,
\]
so, by the corollary $g_h(z,1)\equiv0\Rightarrow g_h(z,t)\equiv0$.\qed

\begin{rem}
The condition on $\alpha$ and $\beta$ is, as it should be expected, weaker than the original condition $\alpha\beta<4$. Actually $\alpha+\beta<2\Rightarrow \alpha\beta<1$. Going back to the example we considered in Section 3 to show that there are nonzero functions satisfying the hypothesis of Theorem 3.1, recall that then $\alpha=1/2$ and $\beta=5/2$, so in that case $\alpha+\beta=3$, hence, that example does not make any contradiction with this result.
\end{rem}

\begin{rem} We can easily see that the condition $\alpha+\beta<2$ is optimal, in the sense that we can give examples of solutions satisfying the hypothesis when $\alpha+\beta=2$. Indeed, consider $f_k^h(0)=\frac{I_k(-i/h^2)}{I_0(1/h^2)}$ and then
\[\begin{aligned}
f_k^h(t)&=\frac{e^{-2it/h^2}}{I_0(1/h^2)}\sum_{m=\infty}^\infty I_m(-i/h^2)I_{k-m}(2it/h^2)\\& \Rightarrow f_k^h(1)=e^{-2it/h^2}\frac{I_k(i/h^2)}{I_0(1/h^2)}.
\end{aligned}\]

Hence, we have a non-zero function that satisfies the decay conditions with $\alpha=\beta=1$. On the other hand, looking at Theorem 3.1 we have that if the function $u_0$ of the hypothesis \eqref{mesh} exists, it has to be zero, and it is as easy to check as in the example in Section 3 that this is the case. 
\end{rem}

Roughly speaking, what this theorem is saying is that in order to have a nonzero solution, if one of the coefficients is small, then the best possible coefficient one can have at time 1 is close to 2, thinking about being "best" as being the smallest possible coefficient, since the smaller the coefficient $\alpha$ or $\beta$, the better the decay. On the other hand what the classical Hardy's principle states is that the smaller one coefficient is, the bigger the other coefficient has to be in order to have a nonzero function.

Considering now the heat equation, we can use the same argument as for the Schr\"odinger equation in order to have a very similar result, again asking about decay conditions at two different times. Now solving the heat equation is equivalent to multiply in the periodic setting by $e^{\cos(xh)}$, so we use Theorem 2.1 with $b=0$ and $\theta=0$ in order to have:

\begin{cor}
Assume that $g_h$ is a complex-valued function and that there are $r,s>0,\delta\in(0,\pi/2)$ such that, for all $y\in\mathbb{R}$,
\[\begin{aligned}
&\left|g_h\left(\frac{\pi/2+\delta}{h}+i y\right)\right|,\left|g_h\left(-\frac{\pi/2+\delta}{h}+iy\right)\right|\le Ce^{-\frac{r}{h^2}\sin\delta\cosh(yh)},\\
&|\hat{g}_h(k)|\le CI_k\left(\frac{1}{sh^2}\right),\ \ \forall k\in\mathbb{Z}.\end{aligned}
\]

Then:
\begin{enumerate}
\item $rs<1\Rightarrow$ There are nonzero functions that satisfy the hypotheses.
\item $rs=1\Rightarrow g_h=Ce^{\frac{r}{h^2}\cos(zh)}$ (and $\hat{g}_h(k)=CI_k(r/h^2)$).
\item $rs>1\Rightarrow g_h\equiv 0.$ 

\end{enumerate}
\end{cor}

And by means of this Corollary, we have the result: 
\begin{teo} Let $v_m^h(t)$ be a solution to the discrete heat equation \eqref{sch}, and assume that $|v_k^h(0)|\le \frac{I_k(\alpha/h^2)}{I_0(\alpha/h^2)},\ |v_k^h(1)|\le \frac{I_k(\beta/h^2)}{I_0(\beta/h^2)}$ and $\alpha$ and $\beta$ are positive numbers satisfying $\alpha+\beta<2,$ then $v^h=(v_k^h)\equiv0$. 
\end{teo}

\begin{proof} The proof of this theorem follows the same argument as Theorem 4.1, just rewriting it for solutions to the discrete heat equation instead of the discrete Schr\"odinger equation.

\end{proof}

Moreover, we can see again that the condition $\alpha+\beta<2$ is optimal, but now we can construct examples that are not discrete versions of the identically zero function and are close to the optimal condition $\alpha+\beta=2$. For instance, if we take the function $u_0(x)=e^{-x^2/2\epsilon}$, then $u(x,t)=e^{t\partial_{xx}}u_0(x)$ is the function given by
\[
u(x,t)=\frac{e^{-x^2/(4t+2\epsilon)}}{\sqrt{2t/\epsilon+1}},
\]
so in this case, we can take $\delta$ on Hardy's uncertainty principle for the heat equation bigger than, but as close as we want to $2+\epsilon$. Then, if we take a discrete version by using the modified Bessel functions, we have a result with $\alpha=\epsilon$ and $\beta=2+\epsilon$, so $\alpha+\beta$ is as close as we want to the optimal value $2$.

\section{Appendix: Examples}

In this Appendix we are going to see some examples of nonzero functions that satisfy the hypotheses of Theorem 2.1 when we are in the first case of the theorem $rs<1$.

\textbf{First example}

The first example, and the most basic example, is given by the function $f_h(z)=Ce^{\frac{v}{h^2}\cos(zh-b)}$, where $v=ae^{i\theta},$ for $r<a<1/s$.

\textbf{Second example}

A calculation shows that the Fourier coefficient of $Ce^{\frac{u}{h^2}\cos(zh-b)}$ is $I_k(u/h^2)e^{-ikb}$. Here we are going to take $u$ such that $|u|=1$ and $a>1$, in order to define the function whose Fourier coefficients are
\[
\hat{f_h}(k)=(ae^{-ibh})^k I_k\left(\frac{u}{a h^2}\right).
\]

This function satisfies the hypotheses for $s=1$ and $\tilde{u}=\frac{u}{a^2}$. Indeed, one can compute the function $f_h$ to get
\[
f_h(z)=e^{\frac{u}{2}(e^{i(zh-b)}+e^{-i(zh-b)}/r^2)},
\]
and now it is easy to check that all the hypotheses are satisfied.

\textbf{More examples}

In some particular cases, like the particular cases of Corollary 4.1 or Corollary 4.2, we can get more examples by studying the evolution of discrete equations. For instance, if we assume that we are in the case of the discrete heat equation, that is, we have $b=\theta=0$, and, for the sake of simplicity, we assume that $h=1$, once we have an example of a function that satisfies the hypotheses, say $f$, we can generate more functions of this type by using the original function as initial datum of the discrete heat equation. Indeed, for the evolution of the Fourier coefficients we have
\[
\hat{f}(k,t)=e^{-2t}\sum_{m\in\mathbb{Z}}\hat{f}(m,0)I_{m-k}(2t)\Rightarrow |\hat{f}(k,t)|\le I_{k}\left(\frac{1}{s}+2t\right).
\]

On the other hand
\[
|f(\pi/2+\delta+iy,t)|=e^{-2t-(2t+r)\sin\delta\cosh(y)},
\]
so the function $f(z,t)$ also satisfy the hypotheses of the corollary, with $r_t=r+2t$ and $s_t=\frac{s}{1+2ts}$. Notice that since $rs<1$, we have immediately that $r_ts_t<1$ for all $t$.

Furthermore, if we have a function not necessarily on the class of functions that satisfy the hypotheses, but it satisfies the condition on the Fourier coefficients, let us say a function $g$ such that for some $a>0$,
\[
|\hat{g}(k)|\le C I_k(a),
\]
and we solve again the discrete heat equation then we have that
\[\begin{aligned}
|\hat{g}(k,t)|&\le CI_k(t+a),\\
|g(t,x+i y)|&\le Ce^{(t\cos(x)+a)\cosh y}.
\end{aligned}\]

If $t\le a$ we do not have the existence of $\delta$ such that the assumptions hold. However, if $t>2a$, there exists $\delta$ such that $t \cos(\pi/2+\delta)=a(\cos(\pi/2+\delta)-1)$ (Notice that $\cos(\pi/2+\delta)=\cos(3\pi/2+\delta)$).

Hence for those $t$ we have that
\[|f(t,\pi/2+\delta+iy)|\le Ce^{a cos(\pi/2+\delta)\cosh y}.\]

So $f(t)$ is in the class of functions which satisfy the assumptions of the corollary. Notice that the product of the two constants is $\frac{a}{t+a}<1$.

On the other hand, we can fix $\delta$ and then, when $t\ge a\frac{\cos(\pi/2+\delta)-1}{\cos(\pi/2+\delta)}$ the hypothesis holds.

\section{Acknowledgments.} This paper is part of the Ph. D. thesis of the author, who is supported by the predoctoral grant BFI-2011-11 of the Basque Government and the projects MTM2011-24054 and IT641-13. The author would like to thank L. Vega, without whose help this paper would not have been possible.%, and the reviewers for their constructive comments that have improved the paper.

%In that particular case, if we have a function $f_h$ that satisfies the hypothesis, and we use its Fourier coefficients as the initial datum of the discrete Schr\"odinger equation, then the solution is going to satisfy again the hypothesis. On the other hand, if we only assume that the Fourier coefficients decay as the modified Bessel function. Then, when we solve the discrete Schr\"odinger equation, we can reach a time where we can take a $\delta$ such that the solution satisfies all the hypothesis.
\end{document}